\documentclass[10pt,twoside]{amsart}
\usepackage{amsopn}
\usepackage{amsthm,siunitx}
\usepackage{amssymb,amsmath,amscd, eurosym,mathtools}
\usepackage{comment}
\usepackage{enumitem}
\usepackage{tikz}

\usepackage{tikz}

\usepackage{cancel}

\usepackage{pifont}
\usepackage{color}

\usetikzlibrary{decorations.pathreplacing}

\newtheorem{theorem}{Theorem}[section]
\newtheorem{lemma}[theorem]{Lemma}
\newtheorem{proposition}[theorem]{Proposition}
\newtheorem{corollary}[theorem]{Corollary}

\theoremstyle{definition}
\newtheorem{definition}[theorem]{Definition}  
\theoremstyle{remark} 
\newtheorem{remark}[theorem]{Remark} 
\newtheorem{example}[theorem]{Example}

\newtheorem{question}[theorem]{Question}

\DeclareMathOperator{\reg}{reg}

\reversemarginpar

\renewcommand{\phi}{\varphi}
\renewcommand{\theta}{\vartheta}
\renewcommand{\epsilon}{\varepsilon}

\begin{document}

\title[Steiner Configurations of Points]{Steiner Systems and  Configurations of Points }
\thanks{Last updated: May 22, 2020}

\author[E. Ballico]{Edoardo Ballico}
\address{Dipartimento di Matematica\\
 Via Sommarive, 14 -  38123 Povo (TN), Italy}
\email{edoardo.ballico@unitn.it}

\author[G. Favacchio]{Giuseppe Favacchio}
	\address{Dipartimento di Matematica e Informatica\\
		Viale A. Doria, 6 - 95100 - Catania, Italy}
	\email{favacchio@dmi.unict.it} \urladdr{www.dmi.unict.it/~gfavacchio}
\author[E. Guardo]{Elena Guardo}
\address{Dipartimento di Matematica e Informatica\\
		Viale A. Doria, 6 - 95100 - Catania, Italy}
	\email{guardo@dmi.unict.it} \urladdr{www.dmi.unict.it/~guardo}
\author[L. Milazzo]{{Lorenzo Milazzo}{$\dag$}}\thanks{$\dag$   Deceased, March 4, 2019.}
\address{Dipartimento di Matematica e Informatica\\
		Viale A. Doria, 6 - 95100 - Catania, Italy}
	\email{milazzo@dmi.unict.it} \urladdr{www.dmi.unict.it/~milazzo}



\begin{abstract} The aim of this paper is to make a  connection between design theory and algebraic geometry/commutative algebra. In particular,   given any Steiner System  $S(t,n,v)$  we  associate two ideals, in a suitable polynomial ring, defining a Steiner configuration of points and its Complement. We focus on the latter, studying its homological invariants, such as Hilbert Function and Betti numbers. We also study  symbolic and regular  powers associated to the ideal defining a  Complement of a Steiner configuration of points, finding  its Waldschmidt constant, regularity,  bounds on its resurgence and  asymptotic resurgence. 
We also compute the parameters of linear codes associated to any Steiner configuration of points and its Complement. 
	
\end{abstract}

\maketitle
\section{Introduction}

Combinatorial design theory is the study of arranging elements of a finite set into patterns (subsets, words, arrays) according to specified rules. It is a field of combinatorics connected to several other areas of
mathematics including number theory and  finite geometries. In the last years, the main techniques used in  combinatorial and algebraic geometry allow  design theory to grow up involving applications in other areas such as in coding theory, cryptography, and computer science. 

A $t-(v,n,\lambda)$-design $D=(V,B)$ is a pair consisting of a set $V$ of $v$  points and a collection $ B$ of $n$-subsets of $V$, called blocks, such that every $t$-subset (or $t$-tuple) of $V$ is contained in exactly $\lambda$ blocks in $B$.


 The numbers $v=|V|$, $b=|B|, n$, $\lambda$, and $t$ are called the {\em parameters} of the design.

A Steiner system $(V,B)$ of type  $S(t,n,v)$  is a $t-(v,n,1)$ design, that is,  a collection $B$ of $n$-subsets (blocks) of a $v$-set $V$ such that each $t$-tuple of  $V$ is contained in a unique block in $B$.  The elements in $V$ are called vertices or points and those of $B$ are called blocks.  In particular,  a Steiner triple system of order $v$, $STS(v)$, is a collection of triples ($3$-subsets) of $V$, such that each unordered pair of elements is contained in precisely one block, and a Steiner quadruple system of order $v$, $SQS(v)$, is a collection of quadruples (4-subsets) of $V$ such that each triple is found in precisely one block. 

A geometric study of some particular classes of Steiner systems, with an eye on the automorphism groups, can be found in \cite{M}. There exists a very extensive literature on STSs and SQS, e.g. \cite{BGMV, Buratti,  Colbourn, ColbournPulleyblank,  ColbournRosa, Lucia, GMRV, GGM,GGMV, LindnerRosa, MZ1, MZ, MZ2, MTV, M1, MRV, ReidRosa} to cite some of them.

In this paper we study special configurations of reduced points in $\mathbb P^n$ constructed on  Steiner Systems, combining Combinatorial  Algebraic Geometry and Commutative Algebra. 
We found many results concerning geometric and algebraic codes and Steiner systems (or designs in general) but  we did not find in the literature any result concerning the study of combinatorial structures, as Steiner systems, using their homological invariants.  Thus, we consider this paper as a starting point of the study of  combinatorial designs, such as Steiner systems, using the homological invariants of their defining ideals. An introduction of coding theory  with the use of Algebraic Geometry can be found in \cite{GeerLint}. In particular, we can refer to \cite{CAG} for some techniques used to study  algebraic varieties with special combinatorial features.
 Recently, connections between commutative algebra and coding theory
have gained much  attention (see \cite{BoPe, Coop-Gu, CSTPV,Toh3, GLS, H1, MBPV, Toh1,Toh2, TohVT}  for more results in this direction). For a background in commutative algebra with a view toward algebraic geometry we suggest \cite{Eis}.


The paper is structured as follow.  Section  \ref{Prelim}  provides the background on designs, ideals of points, symbolic and regular powers of ideals, containment problem and monomial ideals.   We introduce two finite sets of reduced points in $\mathbb P^n$ called \textit{Steiner configuration} and  \textit{Complement of a Steiner configuration} (see Definitions \ref{Steiner} and \ref{C-Steiner}). As pointed out in Remark \ref{stars},  a Steiner configuration of points and its Complement constructed from a Steiner system of type $S(t,n,v)$ are subschemes of a so called {\em star configuration} of $\binom{v }{n}$ points in $\mathbb P^n$.

In Section 3, we will focus on the Complement of a Steiner configuration of points because it is a \textit{proper hyperplanes section} of a monomial ideal that is the Stanley-Reisner ideal of a matroid  (see Theorem \ref{t. DeltaT is a matroid}). Thus, each $m$-th symbolic power of such a monomial ideal is arithmetically Cohen-Macaulay and, after a proper hyperplanes section, it agrees with the $m$-th symbolic power of the ideal defining a Complement of a Steiner configuration (see Proposition \ref{p.same invariants}).  This connection allows us to study the homological invariants of the ideal $I_{X_C}$ of the Complement of  a Steiner configuration.  We start with the description of the initial degree $I_{X_C},$   i.e. we describe the minimum integer $d$ such that $(I_{X_C})_d\neq (0)$ or, equivalently, the least degree $d$  of a minimal generator of $I_{X_C}$.  We also compute its Waldschmidt constant and give  results on the containment problem of  $I_{X_C}$,  i.e. the problem of determining for which $m$ and $d$ the containment $I_{X_C}^{(m)}\subseteq I_{X_C}^d$ holds. Section  \ref{s.Homological invariants} is devoted to the computation of the Hilbert Function, the regularity and graded Betti numbers of  $I_{X_C}.$ In particular, we give some bounds for the asymptotic resurgence and the resurgence of $I_{X_C}.$ 

Unfortunately, the  monomial ideal associated to a Steiner configuration of points does not lead to a Cohen-Macaulay algebra and we cannot apply the same methods as for its Complement.  From a combinatorial point of view, two Steiner systems having the same parameters could have very different properties. Computational experiments give us examples where such differences effect the homological invariants. We pose Questions \ref{Q1} and \ref{Q2} on the behaviour of Steiner configurations of points  to lead us and the interested researcher towards future works.

We end the paper with an application to coding theory. Following \cite{Coop-Gu,Toh3,Toh1, Toh2, TohVT, TohX}, we study special linear codes associated to a Steiner configuration of points  and its Complement.  In particular, their combinatorial structure allows us to compute the Hamming distance of the associated code (Proposition \ref{p.hyp e d} and Theorem \ref{codici}).  


The computer softwares CoCoA \cite{Cocoa} and  Macaulay2 \cite{Macaulay2} were indispensable for all the computations we made.



{\bf Acknowledgement} This project began in April 2018 when E. Guardo and  L. Milazzo met together with the idea of making an interdisciplinary project between their own fields of research, such as Combinatorics and Commmutative Algebra/Algebraic Geometry. This is one of the two projects where L. Milazzo gave his contribution with the idea of involving other mathematicians, such as E. Ballico and G. Favacchio. There is another one, still not published and in progress, that involves other people from University of Catania.  However,  L. Milazzo  became quite ill in February  2019, and after many years battle with his illness fought with huge strengh and courage, unexpectedly he passed away on March 4, 2019 in Catania. Lorenzo gave  ideas to this paper,  and turned his attention to this paper every time he could do it. Unfortunately, Lorenzo  was not able to see the final version of this paper and we feel that
his contributions warrant an authorship.  Lorenzo’s interests and contributions to the topics in this paper should be recognized. Lorenzo is greatly missed to all of us.

The first author was partially supported by MIUR (Italy). G. Favacchio, E. Guardo and L. Milazzo's work has  been supported by the Universit\`{a} degli Studi di Catania, "Piano della Ricerca 2016/2018 Linea di intervento 2". 
All the authors are members of GNSAGA of INdAM (Italy).
We  thank M. Gionfriddo, B. Harbourne, and S. Milici  for their useful conversations on the topic and for their encouragement in finishing the project. We also thank  Zs.~Tuza for his useful comments in the revised version of the paper.

The authors thank the two anonymous referees for their useful comments that improved the previous version of the paper.

\section{Preliminaries and notation}\label{Prelim}
	Steiner systems  play an important role in Design Theory. We use \cite{Colbourn} and \cite{ColbournRosa} as main references for all the background on this topic.
       
	
	The existence of a Steiner system strongly depends on the parameters $(t,n,v)$. For instance if $t=2$ and $n=3$  then  $v\equiv 1,3 \mod(6)$ must hold. 
 There are known necessary conditions for the existence of a Steiner system of type $S(t,n,v)$ that are not in general sufficient.   If a Steiner system $(V,B)$ of type $S(t,n,v)$ exists, then $|B|=\frac{{v \choose t}}{{n \choose t}}.$
	
	\begin{example}\label{e.STS(7)_1}
		  One of the simplest and most known examples of Steiner system is the Fano Plane. It is unique up to isomorphism and it is a  Steiner system of type $S(2,3,7)$ with block set
$$B:=\{\{1,2,3\},\{1,4,5\},\{1,6,7\}, \{2,4,6\}, \{2,5,7\}, \{3,4,7\}, \{3,5,6\}\}.$$
		
	\end{example}
		
	

\subsection{Ideals of  Steiner configuration of points and its Complement}\label{ss.points}


	Let  $V$ be a set of $v$ points and $\mathcal H:=\{H_1, \ldots H_v\}$ be a collection of distinct hyperplanes $H_j$ of $\mathbb P^n$ defined by the linear forms $\ell_j$  for $j=1\dots,n$ and $n\leq v$. Assume that any $n$ hyperplanes in $\mathcal H$ meet in a point.
There is a natural way to associate a point in $\mathbb P^n$ to a $n$-subset of $V.$ Indeed, given a $n$-subset  $\sigma:=\{\sigma_1, \ldots \sigma_n\}$ of $V$, we denote by $P_{\mathcal H, \sigma}$  the point intersection of the hyperplanes $H_{\sigma_1},\ldots, H_{\sigma_n}.$ Then the ideal $I_{P_{\mathcal H, \sigma}}=(\ell_{\sigma_1},\ldots, \ell_{\sigma_n})\subseteq k[\mathbb P^n]$ is  the vanishing ideal of the point $P_{\mathcal H, \sigma}$.


\begin{definition}\label{Steinergeneral} Given  a set $V$  of $v$ points and  a general collection $Y$ of subsets of $V$, we define the following set  of points in $\mathbb P^n$  with respect to $\mathcal H$  $$X_{\mathcal H,Y}:=\cup_{\sigma\in Y}~ P_{\mathcal H, \sigma}$$ and  its defining ideal $$I_{X_{\mathcal H,Y}}:=\cap_{\sigma\in Y} ~  I_{P_{\mathcal H, \sigma}}.$$
\end{definition}

In particular, if $n\leq v$ are positive integers and  $V$ is a set of $v$ points, we denote by $C_{(n,v)}$  the set containing all the $n$-subsets of $V$. 

\begin{definition}\label{Steiner}  Let $(V,B)$ be a Steiner system of type $S(t,n,v)$ with $t<v\leq n$. We associate to $B$ the following set  of points in $\mathbb P^n$   $$X_{\mathcal H,B}:=\cup_{\sigma\in B } ~ P_{\mathcal H, \sigma}$$ and  its defining ideal $$I_{X_{\mathcal H,B}}:=\cap_{\sigma\in B} ~ I_{P_{\mathcal H, \sigma}}.$$
 We call $X_{\mathcal H, B}$ the  Steiner configuration of points associated to the Steiner system $(V,B)$ of type $S(t,n,v)$ with respect to $\mathcal H$ (or  just $X_{ B}$  if there is no ambiguity).
\end{definition}

\begin{definition}\label{C-Steiner}
 Let $(V,B)$ be a Steiner system of type $S(t,n,v)$  with $t<n\leq v$. We associate to $C_{(n,v)}\setminus B$ the following set  of points in $\mathbb P^n$   $$X_{\mathcal H,C_{(n,v)}\setminus B}:=\cup_{\sigma\in C_{(n,v)}\setminus B}~ P_{\mathcal H, \sigma}$$ and  its defining ideal $$I_{X_{\mathcal H,C_{(n,v)}\setminus B}}:=\cap_{\sigma\in C_{(n,v)}\setminus B} ~ I_{P_{\mathcal H, \sigma}}.$$
We call $X_{\mathcal H,C_{(v,n)}\setminus B}$ the  Complement of a Steiner configuration of points  with respect to $\mathcal H $  (or C-Steiner $X_{\text{C}}$ if there is no ambiguity).
\end{definition}
\begin{remark}\label{stars}
Note that from Definition \ref{Steinergeneral}, it follows that $X_{\mathcal H,C_{(n,v)}}$ is a so called {\em star configuration} of ${v \choose n}$ points in $\mathbb{P}^n$. In particular, from definitions \ref{Steiner} and \ref{C-Steiner},  $X_{\mathcal H,C_{(n,v)}\setminus B}$ is the Complement of the Steiner configuration $X_{\mathcal H, B}$ in the star configuration $X_{\mathcal H,C_{(n,v)}}$.  Thus,  a Steiner configuration of points and its Complement are subschemes of a star configuration of ${v \choose n}$  points in $\mathbb{P}^n$. See for instance \cite{CCGVT, CGS, CGVT,GHM, GHMN} for  recent results on  star configurations and see \cite{TohX} for the connections between linear codes and ideals of star configurations.
\end{remark}

\begin{remark}
Since the set  $X_{\mathcal H,B}$  contains $|B|$ points,  we have that 	
$$\deg X_{\mathcal H, C_{(n,v)}\setminus B}= {v \choose n}- |B|= {v \choose n}- \dfrac{{v \choose t}}{{n \choose t}}.$$
\end{remark}

\begin{example}\label{e.STS(7)}
		Consider  the Steiner configuration associated to $(V,B)$ of type $S(2,3,7)$  as in  Example \ref{e.STS(7)_1}. Take $\mathcal H :=\{H_1, \ldots, H_7 \}$  a collection of $7$ distinct hyperplanes $H_i$ in  $\mathbb P^3$   defined by a linear form $\ell_i$ for $i=1,\dots,7$, respectively, with the property that any $3$ of them meet in a point $P_{\mathcal H,\sigma}=H_{\sigma_1}\cap H_{\sigma_2}\cap H_{\sigma_3}$, where $\sigma=\{\sigma_1,\sigma_2,\sigma_3\}\in B$. We get  that $X_{\mathcal H,C_{(3,7)}}$ is a  {\em star configuration} of ${7\choose 3}=35$ points in $\mathbb{P}^3$,  $X_{\mathcal H,B}:=\cup_{\sigma \in B}~ \{P_{\mathcal H, \sigma} \}$ is  a Steiner configuration consisting of $7$ points in $\mathbb P^3$ and  $X_{\mathcal H, C_{(3,7)}\setminus B}$ is  a C-Steiner configuration  consisting of ${7 \choose 3}-7= 28$ points in $\mathbb P^3.$  Their defining ideals are respectively, $$I_{X_{\mathcal H,B}}:=\cap_{\sigma\in B} ~ I_{P_{\mathcal H, \sigma}} \textrm{ and } I_{X_{\mathcal H,C_{(3,7)}\setminus B}}:=\cap_{\sigma\in C_{(3,7)}\setminus B} ~ I_{P_{\mathcal H, \sigma}}.$$
	\end{example} 

\subsection{Symbolic and regular powers of an ideal and the containment problem}\label{ss.symb}
In this section,  we recall some definitions and known results concerning the symbolic and regular powers of ideal of points.

Let $I$ be a homogeneous ideal in the standard graded polynomial ring $R:=k[x_0, \ldots , x_n ].$  Given an integer $m$, we denote by $I^m$ the regular power of the ideal $I$. The $m$-th \textit{symbolic power}  of $I$ is defined as 
$$I^{(m)}=\bigcap_{\mathfrak p\in Ass(I)} (I^mR_{\mathfrak p}\cap R) $$
where $Ass(I)$ denotes the set of associated primes of $I.$
If $I$ is a radical ideal (this includes for instance squarefree monomial ideals and ideals of finite sets of points) then 
$$I^{(m)}=\bigcap_{\mathfrak p\in Ass(I)} {\mathfrak p}^m.$$

It always holds that  $I^{m}\subseteq I^{(m)}.$ In particular, the  \textit{containment problem} is of interest, i.e. the problem of determining for which $m$ and $d$ the containment $I^{(m)}\subseteq I^d$ holds. 
We refer the reader to \cite{Bocci...Vu, BH1, BH, Czapliski...Szpond,  DDGHN, BaczynskaMalara, SzembergSzpond}  for a partial list of papers on this topic. 

The investigation of this problem led to the introduction of other invariants with the aim of comparison between  symbolic and ordinary powers of an ideal.

Let $I$ be a homogeneous ideal, the real number 
$$\rho(I)=\sup\left\{\dfrac{m}{r}: I^{(m)}\not\subseteq I^r \right\}$$
is called \textit{resurgence} of $I$ (see \cite{BH1,BH}).  For a homogeneous ideal $(0)\subset I \subset  k[\mathbb{P}^n ]$, in \cite{GHVT} the authors define an asymptotic resurgence as follows:
$$\rho_{a}(I) = sup\{\frac{m}{r} : I ^{(mt)}\not\subseteq I^{ r t} \text{ for all $t >> 0$}\}.$$

The resurgence is strictly related to the \textit{Waldschmidt constant} of $I$, introduced in \cite{Waldschmidt} in a completely different setting.  If $\alpha(I)$ is the minimum integer $d$ such that $I_d\neq (0),$  then the Waldschmidt constant of $I$,  $\widehat \alpha(I)$, is the real number 	$$\widehat \alpha(I)=  \lim_{m\to \infty}  \dfrac{\alpha (I^{(m)})}{m}.$$

Given $r$ distinct points $P_i \in \mathbb{P}^n$ and non-negative integers $m_i$,
we denote by $Z = m_1P_1 +\dots + m_rP_r \subset \mathbb{P}^n$ the set of fat points defined by $I_Z = \bigcap_i I_{P_i}^{m_i},$ where $I_{P_i}$ is the prime ideal generated by all forms which vanish at $P_i.$ The set $X=\{P_1,\dots,P_ r\}$ is called the support of $Z$. When $m_i=m$ for all $i$, $I_ Z= \bigcap_i I_{P_i}^{m}=I_X^{(m)}.$

If $I$ is the ideal of a set of (fat) points, we denote  by  $I_i$ (resp. $R_i$)  the vector space span in $I$ (resp. $R$) of the forms of degree $i$ in $I$ (resp. $R$).   The  function $H: \mathbb N \to \mathbb N$ such that
$H(i)=\dim_k(R/I)_i= \dim_k R_i-\dim_k I_i$  is called the Hilbert Function of $R/I$ and it is denoted by $H_{R/I}(i)$.
The Castelnuovo-Mumford regularity of $I$, denoted by $\reg(I)$,  is the least degree $t > 0$ such that $\dim(R/I)_t = \dim(R/I)_{t-1}$.

Theorem 1.2.1 in \cite{BH} shows that if $I$ defines a $0$-dimensional scheme, i.e., a finite set of  points, then $$\dfrac{\alpha (I)}{\widehat \alpha (I)} \le \rho(I)\le \dfrac{\reg(I)}{\widehat \alpha (I)}.$$

In \cite{GHVT}, Theorem 1.2.  shows that  $(1)$  $1\leq \dfrac{\alpha (I)}{\widehat \alpha (I)} \le \rho_a (I)\le\rho(I)\le h$ where $h = min(N, h_I )$ and $h_I$
is the maximum of the heights of the associated primes of $I;$ $(2)$  if $I$ is the ideal of a (non-empty) smooth subscheme of $\mathbb{P}^n$, then $\rho_a(I )\leq \frac{\omega(I )}{\widehat \alpha (I)}\leq \frac{\reg(I)}{\widehat \alpha (I)}$ where $\omega(I)$ is the largest degree in a minimal homogeneous set of
generators of $I.$



\subsection{Monomial ideals, simplicial complexes}\label{ss. mon}

In this section, we show how monomial ideals form an important link between commutative algebra and combinatorics.
A monomial ideal is uniquely determined by the monomials it contains. Monomial ideals also arise in graph theory. Given a graph $G$ with vertices $\{x_1,\dots , x_v\}$, we associate the ideal $I_G$ in $k[x_1, \dots , x_v]$ generated by the quadratic monomials $x_ix_j$ such that $x_i$ is adjacent to $x_j$.  From known results in the literature, it is possible to determine many invariants  for monomial ideals. 

In particular,  ideals generated by squarefree monomials have a beautiful combinatorial interpretation in terms of simplicial complexes. In the next sections, we will use these properties to completely describe the most important homological invariants of an ideal defining a Complement of a Steiner configuration of points. This is equivalent to say that we are describing special subsets of  star configurations using tools from combinatorics and algebraic geometry. Unfortunately, the  ideal defining  a Steiner configuration of points cannot be described using monomial ideals, and we cannot determine its homological invariants. So, the problem is still open. At the end of the paper, we will be able to describe the parameters of the linear codes associated to both a Steiner configuration of points and its Complement.

We refer to \cite{HH} for notation and basic facts on monomial ideals and to \cite{BH} for an extensive coverage of the theory of Stanley-Reisner ideals.
\begin{definition} 
An ideal $I$ in a polynomial ring $R$ is called a monomial ideal if there is a subset $\mathcal A\subset \mathbb{Z}^n_{\geq 0}$ (possible infinite) such that $I$ consists of all polynomials of the form $\sum_{\alpha\in \mathcal A} h_{\alpha} x^{\alpha}$, where $h_{\alpha}\in R$.
\end{definition}

\begin{definition} A simplicial complex $\Delta$ over
a set of vertices $ V=\{x_1,\dots,x_v\}$ is a collection of subsets of $V$ satisfying the following two conditions:
\begin{enumerate}
\item $\{x_i\}\in \Delta$ for all $1\leq i\leq v$
\item if $F \in \Delta$ and $G\subset F$, then $G\in \Delta$.
\end{enumerate}
\end{definition}

An element $F$ of $\Delta$ is called a face, and the dimension of a face $F$ of $\Delta$ is  $|F|-1$, where $|F|$ is the number of vertices of $F$.
The faces of dimensions $0$ and $1$ are called vertices and edges, respectively, and $\dim \emptyset = -1$.

The maximal faces of $\Delta$ under inclusion are called facets of $\Delta$. 
The dimension of the simplicial complex $\Delta$ is $\dim \Delta = \max\{\dim F~ |~ F  \in \Delta\}$. We refer to $i$-dimensional faces as $i$-faces.
We denote the simplicial complex $\Delta$ with facets $F_1,\dots , F_q$  by
$\Delta = \langle F_1,\dots, F_q\rangle
$ and we call $\{F_1,\dots , F_q\}$ the facet set of $\Delta$.

If  $\Delta$ is a $d$-dimensional simplicial complex, the most important invariant  is the  $f$-vector (or  {\em face vector}) of  $\Delta$ and it is denoted by $(f_0,\dots , f_d) \in \mathbb N^{d+1}$, where $f_i$ denotes the number of $i$-dimensional
faces in $\Delta$. From the monomial ideal point of view, the $f$-vector is encoded in the Hilbert series of the quotient ring $R/I_{\Delta}$ and it is related to the $h$-vector of a suitable set of points in $ \mathbb P^n$ (see Section \ref{s.Homological invariants}). We will use Remark \ref{r. I_S Stanley-Reisner} to show the connections  between Stanley-Reisner ideals, simplicial complexes, matroids and Steiner systems to determine   Waldschmidt constant  and bounds for the resurgence (see Section 3),  Betti numbers, Hilbert function, and regularity of the ideal of a Complement of a Steiner configurations of points (see Section 4).

The Alexander dual of a simplicial complex $\Delta$ on $V=\{x_1,\dots,x_v\}$ is the simplicial complex $\Delta^{\vee}$ on $V$ with faces $V\setminus \sigma$, where $\sigma\not\in \Delta$.

The Stanley-Reisner ideal of  $\Delta$ is the ideal $ I_\Delta :=(x^{\sigma} ~ |~ \sigma \not\in \Delta)$ of $R = k[x_1,\dots , x_v]$,  where $x^{\sigma} =\Pi_{i \in \sigma} ~  x_i$.  It is well known that the Stanley-Reisner ideals are precisely the squarefree monomial ideals. The quotient ring $k[\Delta] := R/I_{\Delta}$
is the Stanley-Reisner ring of the simplicial complex $\Delta$.

Each simplicial complex has a geometric realization as a certain subset  of a finite dimensional affine space.

If  $ V$  is a set $v$ points, we denote by $k[V]:=k[x_1,\ldots, x_v]$ the standard graded polynomial ring in $v$ variables. 
Given a $n$-subset of $V$, $\sigma:=\{{i_1}, {i_2}, \ldots, {i_n}\}\subseteq V,$ we will write
$$\mathfrak p_{\sigma}:=(x_{i_1}, x_{i_2}, \ldots, x_{i_n})\subseteq k[V]$$
for the prime ideal generated by the variables indexed by $\sigma$, and
$$M_{\sigma}:=x_{i_1} x_{i_2} \cdots x_{i_n}\in  k[V]$$
for the monomial given by the product of the variables indexed by $\sigma$.




Let $n\leq v$ be positive integers, and $V$ a set of $v$ points; recall that $C_{(n,v)}$ is  the set containing all the $n$-subsets of $V$. 

\begin{definition}  
	If  $T\subset C_{(n,v)}$,  we define two ideals
$$I_T:= (M_{\sigma} \ |\ \sigma\in T)\subseteq k[V]$$   and
$$J_T:= \bigcap_{\sigma\in T} \mathfrak p_{\sigma} \subseteq k[V]$$   
called the {\em face ideal} of $T$ and the {\em  cover ideal} of $T,$ respectively.
\end{definition}

\begin{remark}\label{r. I_S Stanley-Reisner} It is well known that $J_T$ is the Stanley-Reisner ideal $I_{\Delta_T}$ of the simplicial complex 
	$$\Delta_T:=\left\langle V\setminus \sigma | \sigma \in T\right\rangle.$$  Then $J_T$ is generated by the monomials $M_b$ with $b\notin \Delta_T.$ We also recall that $I_T$ and $J_T$ are  the Alexander duals of each other.	 
\end{remark}

Since $J_T$ is a squarefree monomial ideal, the $m$-th symbolic power of $J_T$ (Theorem 3.7 in \cite{CooperHa}) is
$$J_T^{(m)}:= \bigcap_{\sigma\in T} \mathfrak p_{\sigma}^m.$$

\section{Matroid and  Configurations of points from  Steiner systems: Waldschmidt constant and containment problem}\label{s.Matroid} 


%
%

This section is devoted to show how the results of the previuos sections are related to our special configurations of points. In particular, we will show that the Complement of a Steiner configuration of points is connected with the theory of matroids.

\begin{definition}	
A simplicial complex $\Delta$ is said to be a matroid if  $F,G\in \Delta$ and  if  $|F|>|G|$ then there exists $i\in F\setminus G$ such that $G\cup\{i\}\in \Delta$.	
\end{definition}

We also  recall that we say that a homogeneous ideal $J$ in a polynomial ring $R$ is \textit{Cohen-Macaulay} (CM) if $R/J$ is Cohen-Macaulay,  i.e. $\text{depth}(R/J) = \text{Krull-dim}(R/J).$  Varbaro in \cite{Varbaro}  and Minh and Trung in \cite{MT} have independently shown the following

\begin{theorem}[Varbaro \cite{Varbaro}, Minh and Trung \cite{MT}]\label{t.CM Matroid} Let $\Delta$ be a simplicial complex. Then
	$k[V]/I_{\Delta}^{(m)}$ is Cohen-Macaulay for each $m\ge 1$ if and only if $\Delta$ is a matroid. 
\end{theorem}

Terai and Trung in \cite{TT} proved if $I_{\Delta}^{(m)}$ is Cohen-Macaulay for some $m\ge 3$  then $\Delta$ is a matroid.
	

To shorten the notation,  from now on we set   $C:=C_{(n,v)}\setminus B$. Let $X_{\mathcal H,B}$ be a Steiner configuration of points and $X_{\mathcal H, C}$ its  Complement as previous defined with respect to a collection $\mathcal H$ of hyperplanes in $\mathbb P^{n}.$ 

We claim that $I_{X_{\mathcal H, C}}^{(m)}$ and $J_C^{(m)}$ share the same homological invariants for any positive integer $m.$ 
In particular, the key point of our argument is that $J_C$ is the Stanley-Reisner ideal of the simplicial complex $\Delta_C$, and  that $\Delta_C$ is a matroid (see Theorem \ref{t. DeltaT is a matroid}).

Several times in the literature  simplicial complexes have been associated to Steiner systems. See for instance \cite{ColbournPulleyblank, Swartz}. In Example 4.6 in \cite{DDGHN} the well known Fano matroid is used to construct a Cohen-Macaulay ideal $I$ such that $I^{(3)}\neq I^3$ and $I^{(2)}\neq I^2.$ The quoted example is one of the cases described in Corollary \ref{c. containement points}.  


We need the following auxiliary lemma.
\begin{lemma}\label{l .C_(v-n-1) subseteq Delta_T}
	Let $(V,B)$ be a Steiner system of type $S(t,n,v).$  If  $C_{(v-n-1,v)}$
is  the set containing all the $(v-n+1)$-subsets of $V$, then $C_{(v-n-1,v)}\subseteq \Delta_C$.
\end{lemma}
\begin{proof}
	Let $F$ be a $(v-n-1)$-set of $V.$  In order to prove the lemma we need to find  $G\in C$ such that $F \subseteq V\setminus G,$ i.e., $G \subseteq V\setminus F.$ 
	First note that $|V\setminus F|=n+1$, so we can take two $n$-subsets $G_1, G_2$ of $V\setminus F$ sharing a $n-1$-subset. Since $n-1\ge t$, from the definition of Steiner Systems at least one of them, say $G_1,$ does not belong to $G$. 
	Then $G_1\in C$ and $G\subseteq V\setminus F.$ 
\end{proof}

The next results are useful to describe the Complement of a Steiner configuration of points using the combinatorial properties of matroids.

\begin{theorem}\label{t. DeltaT is a matroid}
	Let $(V,B)$ be a Steiner system of type $S(t,n,v).$ Then $\Delta_C$ is a matroid. 
\end{theorem}
\begin{proof}
	Let $F, G$ be two maximal elements in $\Delta_C$, i.e., $F,G$ belong to the facet set of $\Delta_C$ .
	Assume by contradiction  $(F\setminus \{i\})\cup \{j\}\notin \Delta_C$ for each $i\in F$ and for each $j\in F$. Then all the sets $\sigma_{i,j}:=V\setminus ((F\setminus \{i\})\cup \{j\})$ belong to $B$. Since $n\ge 2$,  at least two blocks of $B$ share a $(n-1)$-set. That is a contradiction because $n-1>t-1$ from Definition \ref{Steiner}. 
\end{proof}

\begin{proposition}\label{matr}
	Let $(V,B)$ be a Steiner system of type  $S(t,n,v).$
	Then $k[V]/J_C^{(m)}$ is Cohen-Macaulay for each $m\ge 1.$ 
\end{proposition}
\begin{proof} Since $J_C=I_{\Delta_C}$, it is a consequence of Theorem \ref{t.CM Matroid} and Theorem 	\ref{t. DeltaT is a matroid}.
\end{proof}

Now, given a Steiner system of type $S(t,n,v)$ we  are able to describe some homological invariants of the ideal associated to a Complement of Steiner configuration of points. Unfortunately, we have no similar result as Proposition \ref{matr} that should hold for a Steiner configuration of points.

We now prove the main claim of this section. 

\begin{proposition}\label{p.same invariants}
$I_{X_{\mathcal H,C}}^{(m)}\subseteq k[\mathbb P^n]$ and $I_{\Delta_C}^{(m)}\subseteq k[V]$ share the same homological invariants.
\end{proposition}
\begin{proof}
	It is an immediate consequence of Theorem 3.6. in \cite{GHMN}.
	Indeed, $k[V]/I_{\Delta_C}^{(m)}$ is Cohen-Macaulay, and any subset of at most $n$  linear forms in $\{\ell_1, \ldots, \ell_v\}$  is a $k[\mathbb P^n]$-regular sequence.  
\end{proof}

The Cohen-Macaulay property of $k[V]/I_{\Delta_C}^{(m)}$ also allows us to look at  $I_{X_{\mathcal H,C}}^{(m)}$ as a proper hyperplane section of $I_{\Delta_C}^{(m)}$ (see \cite{M}). This construction is quite standard but is very useful to describe combinatorial properties for arithmetically Cohen-Macaulay varieties $X$, i.e, $\text{depth}(R/I_X) = \text{Krull-dim}(R/I_X)$, not only in projective spaces, but also in multiprojective spaces, see for instance \cite{FGM, FGP, FM, GHMN}. The common idea is to relate, when possible, ideals of  arithmetically Cohen-Macaulay varieties to monomial ideals in order to study their invariants.

 We start with the description of the initial degree $\alpha(J_C)$ of $J_C=I_{\Delta_C},$   i.e. we describe the minimum integer $d$ such that $(J_C)_d\neq (0)$ or, equivalently, the least degree $d$  of a minimal generator of $J_C$ . For the ease of the reader, we prove the following result that shows that $\{\alpha(J_C^{m})\}_m$ is a strictly increasing sequence. 
\begin{lemma}\label{l.partial}
	Let $\mathfrak p\subseteq k[V]$ be a squarefree monomial ideal and $M\in \mathfrak p^m$. Then $\dfrac{\partial M}{\partial x_j}\in \mathfrak p^{m-1}$ for any $j\in V$.
\end{lemma}
\begin{proof}
	Set $\mathfrak p=\mathfrak p_{\sigma}$ where $\sigma \subseteq V$ and let  $x_1^{a_1}\cdots x_v^{a_v}\in \mathfrak p^m$. Then $\sum\limits_{i\in \sigma}{a_i}\ge m.$    If $a_j=0$ then trivially $0\in \mathfrak p^{m-1}.$ If $a_j>0$ then $\dfrac{\partial M}{\partial x_j}=a_j x_1^{a_1}\cdots x_j^{a_j-1}\cdots x_v^{a_v}$ and $\sum\limits_{i\in \sigma}{a_i}\ge m-1.$ 
\end{proof}

\begin{proposition}\label{p. alpha(J_T^m)}
	Let $(V,B)$ be  a Steiner system of type $S(t,n,v)$. Then
	\begin{itemize}
		\item[i)] $\alpha(J_{C})=v-n$;
		\item[ii)] $\alpha(J_{C}^{(q)})=v-n+q$ for $2\le q<n$;
		\item[iii)] $\alpha(J_{C}^{(m)})=\alpha(J_{C}^{(q)})+pv$, where $m=pn+q$ and $0\le q<n$.  
	\end{itemize} 
\end{proposition}
\begin{proof}
	\begin{enumerate}
		\item[$(i)$] We have $J_C=I_{\Delta_C}=(M_{\sigma} \ |\ \sigma \not\in \Delta_C)$. From Lemma \ref{l .C_(v-n-1) subseteq Delta_T} we have that the elements not in $\Delta_C$ have cardinality at least $v-n.$
		On the other hand, for any $\beta\in B$ we have $V\setminus \beta\not\in \Delta_C$.  
		
		\item[$(ii)$]   First we show $\alpha(J_{C}^{(q)})\le v-n+q$. Let $\beta\in B$ and $\sigma\subseteq \beta$ with  $|\sigma|=q$, then we claim $M_{(V\setminus \beta)\cup \sigma}\in \alpha(J_{C}^{(q)})$.  Assume by contradiction  $M_{(V\setminus \beta)\cup \sigma}\notin \mathfrak p_{\tau}^q$ for some $\tau\in V.$ Then $(V\setminus \beta)\cap \sigma \cap \tau= \emptyset$ but this is a contradiction since $v-n+q+n>n.$ 	On the other hand, if $q>2$ the statement follows by Lemma \ref{l.partial}, since $\alpha(J_{C}^{(q)})>\alpha(J_{C}^{(q-1)})= v-n+q-1$. If $q=2$ we proceed by contradiction. Let $M\in J_{C}^{(2)}$ be a monomial such that $\deg M = v-n+1.$ Then there are at least $n-1$ variables that do not divide $M$. Since $t\le n-1$, there is at most one block in $B$ containing these variables. Therefore there are at least $v-n$ elements of $C$ containing them.
		Since $M\in J_{C}^{(2)}$,  $M$ belongs to each of these ideals to the power of 2, i.e., there are $v-n$ variables $x_j$ such that $x_j^2|_M.$ So $\deg M \ge 2(v-n)> v-n+1.$  
		
		\item[$(iii)$]  We show that $J_C^{(m)}:M_V=J_C^{(m-n)}.$ 
		\begin{itemize}
			\item If $F\in J_{C}^{(m-n)}$ then $M_VF\in {J_C^{(n)}}\cdot J_{C}^{(m-n)}\subseteq J_{C}^{(m)}$. 
			\item If $FM_V\in J_{C}^{(m)}$ then, by Lemma \ref{l.partial}, $FM_{V\setminus \sigma}\in J_{C}^{(m-n)}$ for any $\sigma\in T$.
			Thus  $FM_{V\setminus \sigma}\in \mathfrak p_{\sigma}^{(m-n)}$. This implies $F\in \mathfrak p_{\sigma}^{(m-n)}$. 
		\end{itemize}
		Then $\alpha(J_{C}^{(m)})= \alpha(J_{C}^{(m-n)})+v =\cdots =\alpha(J_{C}^{(q)})+pv$, where $m=pn+q$ and $0\le q<n$. 
	\end{enumerate}		
\end{proof}


	
From Proposition  \ref{p.same invariants} and Proposition  \ref{p. alpha(J_T^m)}, the initial degree of the ideal of a C-Steiner configuration of points  only depends on the parameters $(t,n,v)$ of the Steiner system. 


Using the previous results, we have the following theorem:
	\begin{theorem}\label{t. alpha(I_X^m)}
		Let $(V,B)$ be a Steiner system  of type $S(t,n,v)$. Then
		\begin{itemize}
			\item[i)] $\alpha(I_{X_C})=v-n$;
			\item[ii)] $\alpha(I_{X_C}^{(q)})=v-n+q$, for $2\le q <n;$
			\item[iii)] $\alpha(I_{X_C}^{(m)})=\alpha(I_{X_C}^{(q)})+pv$, where $m=pn+q$ and $0\le q<n$ and  $\alpha(I_{X_C}^{(n)})=\alpha(I_{X_C}^{(0)})+v=v.$ 
		\end{itemize} 
	\end{theorem}

 An important consequence  of Theorem \ref{t. alpha(I_X^m)} is related to the containment problem.
	
	\begin{corollary}\label{c. containement points} Let $(V,B)$ be a Steiner system of type $S(t,n,v)$. Then $I_{X_C}^{(m)}\not\subseteq I_{X_C}^{d}$ for any pair $(m,d)$ such that 
		$$
		\begin{array}{cccc} m\equiv 1 \mod n &\ \text{and}\ & \ d>1+\dfrac{(m-1)v}{n(v-n)}& \text{or}\\
		&\\
		m\not \equiv 1 \mod n &\ \text{and}\ &\ d>1+\dfrac{m-n}{n}+\dfrac{m}{(v-n)}& \\
		\end{array}$$
		In particular, if $v>2n$ then
		$$I_{X_C}^{(n)}\not\subseteq I_{X_C}^{2}$$
	\end{corollary}
	\begin{proof}
		From item $i)$ in Theorem \ref{t. alpha(I_X^m)} $\alpha(I_{X_C}^d)=d(v-n)$. Therefore, it is enough to take $m$ such that $\alpha(I_{X_C}^{(m)})<d(v-n)$.
		If $m\equiv 1 \mod n$ then, from items $i)$ and $iii)$ in Theorem \ref{t. alpha(I_X^m)}, we get  $\alpha(I_{X_C}^{(m)})=(v-n)+\dfrac{(m-1)v}{n}$.
		Then $(v-n)+\dfrac{(m-1)v}{n}<d(v-n)$ implies the statement.
		If $m\not\equiv 1 \mod n$ then $m=pn+q$ and $2\le q\le n$ and, from items $ii)$ and $iii)$ in Theorem \ref{t. alpha(I_X^m)}, we get 
		$\alpha(I_{X_C}^{(m)})=v-n+m-pn+pv= (1+p)(v-n)+m.$
		Thus $(1+p)(v-n)<d(v-n)+m$ implies $d>1+p+\dfrac{m}{(v-n)}\ge 1+\dfrac{m-n}{n}+\dfrac{m}{(v-n)}.$ 
		
		In the case $m=n$,  we get $1+\dfrac{m-n}{n}+\dfrac{m}{(v-n)}= 1+\dfrac{n}{(v-n)}<2.$
	\end{proof}
	
	\begin{example}
		If $X:=X_{\mathcal H, C_{(3,7)}\setminus B}\subseteq\mathbb P^3$ is the C-Steiner configuration of points as  in Example \ref{e.STS(7)}, then $I_{X}^{(3)}\not\subseteq {I_{X}}^{2}$.
	\end{example}
	
	Another immediate corollary allows us to compute the Waldschmidt constant of a C-Steiner configuration.  The Waldschmidt constant of a uniform matroid was computed in Theorem 7.5 in \cite{Bocci...Vu}. Properties of a uniform matroid were also studied in \cite{GHMN}.
	\begin{corollary} \label{wald}  If $(V,B)$ is a Steiner system of type  $S(t,n,v)$,  then the Waldschmidt constant of $I_{X_C}$ is
		$$\widehat{\alpha}(I_{X_C})= \dfrac{v}{n}.$$
	\end{corollary}
	\begin{proof}
		From \cite{Chudnovsky}, Lemma 1, the limit
		$\widehat{\alpha}(I_{X_C})=\lim\limits_{m\to +\infty} \dfrac{\alpha(I_{X_C}^{(m)})}{m}$ exists. Then from Theorem \ref{t. alpha(I_X^m)} we have  
		$$\widehat{\alpha}(I_{X_C})=\lim\limits_{p\to +\infty} \dfrac{\alpha(I_{X_C}^{(pn)})}{pn}=\lim\limits_{p\to +\infty} \dfrac{pv}{pn} =\dfrac{v}{n}.$$
	\end{proof}

%

\begin{remark}
	The homological invariants do not depend on the choice of the hyperplanes, provided that we take them meeting properly.
\end{remark}



\section{Homological invariants of C-Steiner configurations of points: Hilbert function and graded Betti numbers}\label{s.Homological invariants}

In this section we describe the Hilbert function of a C-Steiner configuration of points. We recall some known definitions. For a finite set of points $X\subset \mathbb P^n$ the Hilbert function of $X$ is defined as the numerical function $H_X: \mathbb N \to \mathbb N$ such that
\[
H_X(i)=\dim_k(R/I_X)_i= \dim_k R_i-\dim_k(I_X)_i
\]
 \noindent where $R=k[\mathbb P^n]$ and  the first difference of the Hilbert function is defined by $\Delta H_X(i):=H_X(i)-H_X(i-1).$ The {\em $h$-vector} of $X$ is  denoted by
\[
h_X= h = (1,h_1,\ldots, h_p)
\]
where $h_i  = \Delta H_X(i)$ and $p$ is the last index such that $\Delta H_X(i)  > 0$. If $I$ is the ideal of a set of  points, analogously to the case of fat points, $\reg(I)$ denotes the Castelnuovo-Mumford regularity of $I.$
%

The next result gives us  informations on the $h$-vector of a C-Steiner configuration of points $X_C$. Recall that given a Steiner system $(V,B)$ of type $S(t,n,v)$, the number of blocks is  $|B|=\dfrac{{v\choose t}}{{n\choose t}}$ and that the last entry $h_p$ of $h_{X_{C}}$ is in degree $p=v-n.$ With the previous  notation, we have the following:
\begin{proposition}\label{p.h-vector}
	If $(V,B)$ is a Steiner system $S(t,n,v)$, then the $h$-vector of $X_C$ is
	$$h_{X_C}=\left(1, n, {n+1 \choose n-1}, \cdots, {v-2 \choose n-1}, {v-1 \choose n-1}-|B| \right) .$$
\end{proposition}
\begin{proof}
	From  Theorem  \ref{t. alpha(I_X^m)} we have $\alpha(I_{\Delta_C})$ is $v-n$. 
	From Remark \ref{r. I_S Stanley-Reisner},  a set of minimal generators of degree $v-n$ has $|B|$ elements.   So we only need to show that $h_{X_C}(v-n+1)=0.$  
	This follows since $\sum_{j=0}^{v-n}h_{X_C}(j)= {v \choose n}-|B|= \deg X_C.$ 
\end{proof}

The regularity of  $I_{X_C}$ is an easy consequence of Proposition \ref{p.h-vector}:
\begin{corollary}\label{c.reg}
	$\reg(I_{X_C})=\alpha(I_{X_C})+1=v-n+1$.
\end{corollary}

Let $(V,B)$ be  a Steiner system of type $S(t,n,v)$. A  minimal graded free resolution of $I_{X_C}$ will be written as 
\begin{equation}\label{res}
0 \rightarrow \bigoplus_j R(-j)^{\beta_{n-1,j}(I_{X_C})}\rightarrow \dots\rightarrow \bigoplus_j R(-j)^{\beta_{1,j}(I_{X_C)}}\rightarrow  \bigoplus_j R(-j)^{\beta_{0,j}(I_{X_C})}\rightarrow R \rightarrow R/I_{X_C} \rightarrow 0
\end{equation}
\noindent where $ R(-j)$ is the free $R$-module obtained by shifting the degrees of $R$ by $j,$
i.e. so that $R(-j)_a = R_{a-j}.$
 
The number $\beta_{i,j}(I_{X_C})$ is called the $(i,j)$-th graded Betti numbers of $I_{X_C}$  and equals the number of minimal generators of degree $j$ in the $i$-th syzygy module of $I_{X_C}.$

Another consequence of Proposition \ref{p.h-vector} 
is that $$\beta_{i,j}(I_{X_C})=0 \textrm{ for any } j-i>v-n+1.$$
This means that the nonzero graded Betti numbers only occur in two rows of the Betti table $\beta (I_{X_C}):=\left (\beta_{i,i+j}(I_{X_C})\right )_{i,j}$ of $I_{X_C}$.  

The next proposition excludes the existence of first syzygies in degree $v-n+1.$ 
\begin{proposition}\label{p. Beta 1 v-n+1}
	If $(V,B)$ is  a Steiner system of type $S(t,n,v)$, then $$\beta_{1, v-n+1}(X_C)=0.$$
\end{proposition}
\begin{proof}
	From Proposition \ref{p.same invariants} it is enough to show that $\beta_{1, v-n+1}(J_C)=0$. By contradiction, let $\alpha_1,\alpha_2\in B$ be two distinct blocks of the Steiner system   $(V,B)$ such that the monomials $M_1:=M_{V\setminus \alpha_1}$ and $M_2:=M_{V\setminus \alpha_2}$ give a linear syzygy. Then $\gcd(M_1,M_2)$ has degree $v-n-1.$ This implies $|\alpha_1\cap \alpha_2|=n-1\ge t,$ a contradiction to Definition \ref{Steiner} of Steiner system. 
\end{proof}

\begin{proposition}\label{p. Beta 0 v-n+1}	If $(V,B)$ is  a Steiner system of type $S(t,n,v)$, then
	$$\beta_{0,v-n+1}(I_{X_C})=n|B|- {v \choose n-1}=n\dfrac{{v\choose t}}{{n\choose t}}- {v \choose n-1}.$$ 	
\end{proposition}

\begin{proof}
	It is matter of computation to show that
	$$\Delta h_{X_C}(v-n+1)= -{v-1 \choose n-1}+|B| $$
 and
	$$\Delta h_{X_C}(v-n)= {v-1 \choose n-1}-|B|- {v-2 \choose n-1}= {v-2 \choose n-2}-|B|   $$
	and for $i<v-n$   
	$$\Delta h_{X_C}(i)= {i+n-2 \choose n-2}.$$
	Then 
	$$\Delta^2 h_{X_C}(v-n+1)=- {v-2 \choose n-2} -{v-1 \choose n-1}+2|B| $$ and
	$$\Delta^2 h_{X_C}(v-n)= {v-3 \choose n-3}-|B|    $$
	and for $i<v-n$   
	$$\Delta^2 h_{X_C}(i)= {i+n-3 \choose n-3}.$$
	So we have
	
	$$\Delta^{n} h_{X_C}(v-n+1)= n|B|-\left[ {v-1 \choose n-1}+ {v-2 \choose n-2}+\cdots {v-n \choose 0}  \right]=$$
	$$=n\dfrac{{v\choose t}}{{n\choose t}}- {v \choose n-1}.$$	
	Then the statement follows from Proposition \ref{p. Beta 1 v-n+1}.
\end{proof}

\begin{corollary}\label{p. Delta n} If $(V,B)$ is  a Steiner system of type $S(t,n,v)$, then
	$\beta_{0,v-n+1}(I_{X_C})=0$ if and only if $t=n-1.$

\end{corollary}
\begin{proof}
	From Proposition \ref{p. Beta 0 v-n+1} we have 
	$$\beta_{0,v-n+1}(I_{X_C})=\Delta^{n} h_{X_C}(v-n+1)= n|B|-\left[ {v-1 \choose n-1}+ {v-2 \choose n-2}+\cdots {v-n \choose 0}  \right]=$$
	$$=n\dfrac{{v\choose t}}{{n\choose t}}- {v \choose n-1}.$$	
	
	If $t=n-1$ then we get
	$$\beta_{0,v-n+1}(I_{X_C})=n\dfrac{{v\choose n-1}}{{n\choose n-1}}- {v \choose n-1}=0.$$	
	If $\beta_{0,v-n+1}(I_{X_C})=0$ then we have
	$$\dfrac{(n-t)!}{(v-t)!}= \dfrac{1}{(v-n+1)!}$$
	and then 
	$$ v-n+1={v-t \choose n-t}.$$
	This implies 
	$$ {v-n+1 \choose v-n}={v-t \choose v-n}$$
	and then $t=n-1.$
	
\end{proof}

\begin{corollary}	If $t=n-1$,  we have
	
	$$\beta_{i,j}(I_{X_C})=\begin{cases}
	|B| & \text{if}\ (i,j)=(0,v-n)\\
	(-1)^i\Delta^{n} h_{X_C}(v-n+1+i) & \text{if}\ j=v-n+1+i\\
	0 & \text{otherwise.}\\
	\end{cases} $$	
\end{corollary}

\begin{corollary}\label{omega}	We have
	
	$$\omega(I_{X_C})=\begin{cases}
	\alpha(I_{X_C})=v-n& \text{if}\ t=n-1\\
	\alpha(I_{X_C})+1=\reg(I_{X_C})=v-n+1& \text{if}\ t<n-1\\
	\end{cases} $$
where  $\omega(I_{X_C})$ is the largest degree in a minimal homogeneous set of
generators of $I_{X_C}.$	
\end{corollary}

Recalling that the resurgence and the asymptotic resurgence of $I$ are defined as $\rho(I):=\sup\left \{\dfrac{m}{d} \ | \ I^{(m)}\not\subseteq I^{d}\right\}$  and $\rho_{a}(I) = \sup\left\{\frac{m}{r} : I ^{(mt)}\not\subseteq I^{ r t} \text{ for all $t >> 0$}\right\}$, respectively, the following results give the  bounds:

\begin{corollary} Let $(V,B)$ be a Steiner system of type $S(t,n,v)$. 
\begin{enumerate}
\item If $t=n-1$ then 

	$$\dfrac{(v-n)n}{v}= \rho_a(I_{X_C})\le \rho(I_{X_C})\le\dfrac{(v-n+1)n}{v}.$$

\item If $t<n-1$ 
$$\dfrac{(v-n)n}{v} \leq \rho_a(I_{X_C})\le \rho(I_{X_C})\le\dfrac{(v-n+1)n}{v}.$$
\end{enumerate}
\end{corollary}
\begin{proof}
	By Theorem 1.2.1 in \cite{BH}  and by Theorem \ref{t. alpha(I_X^m)} and Corollary \ref{omega} we have
	$$\dfrac{(v-n)n}{v}=\dfrac{{\alpha}(I_{X_C})}{\widehat{\alpha}(I_{X_C})}\le \rho(I_{X_C})\le \dfrac{reg(I_{X_C})}{\widehat{\alpha}(I_{X_C})}=\dfrac{(v-n+1)n}{v}$$
	
\end{proof}

\begin{remark}
	From Lemma 2.3.4 in \cite{BH}, if $d\reg(I_{X}) \le \alpha(I_{X}^{(m)})$, then $I_{X}^{(m)} \subseteq I^d$. Moreover, if $d\alpha(I_X)>\alpha(I_X^{(m)})$ then $I_X^{(m)} \not\subseteq I_X^d$. The range of values not covered by these bounds has length 
	$$\left\lfloor\dfrac{\alpha(I_X^{(m)})}{\alpha(I_X)}-\dfrac{\alpha(I_X^{(m)})}{\reg(I_X)}\right\rfloor=\left\lfloor\dfrac{\alpha(I_X^{(m)})(\reg(I_X)- \alpha(I_X))}{\alpha(I_X)\reg(I_X)}\right\rfloor$$
	For a C-Steiner configuration $X_{C}$, from Corollary \ref{c.reg}, this number is equal to
	$$\left\lfloor\dfrac{\alpha(I_{X_C}^{(m)})}{\alpha(I_{X_C})\reg(I_{X_C})}\right\rfloor. $$

\end{remark}


\begin{example}
	We compute the $h$-vector and the graded Betti numbers of a C-Steiner System $X_C$ of type $S(2,3,v)$. We have $|B|=\frac{v(v-1)}{6}$ and then
	$$h_{X_C}=(1, 3,6, \cdots, {v-2 \choose 2}, {v-1 \choose 2}-|B| )=$$ 
	$$\Delta^3 h_{X_C} =(1, 0,0, \cdots, 0, -|B|,0,2{v-1 \choose 2}+(v-2)-3|B|,-{v-1 \choose 2}+|B| )=$$
	$$\ =\left(1, 0,0, \cdots, 0, -\frac{v(v-1)}{6},0,\frac{v(v-3)}{2},-\frac{(v-1)(v-3)}{3}\right).$$ 
In this case, we have: $\beta_{0,v-3}= \frac{v(v-1)}{6},$  $\beta_{1,v-2}= 0,$  $\beta_{1,v-1}= -\frac{v(v-3)}{2}$ and $\beta_{2,v}= -\frac{(v-1)(v-3)}{6}.$ In particular, $\alpha(I_{X_C})=\omega((I_{X_C})=v-3$ and $\reg(I_{X_C})=v-2.$

\end{example}

In the following example we compute the graded Betti numbers of a C-Steiner configuration of type $S(t,n,v)$ where $t< n-1.$
\begin{example}
	Consider a Steiner system $(V,B)$ of type $S(2,4,13)$ where $|V|=13$ and  

\begin{eqnarray*}
B&:=\{ \{2,  3,  5, 11\},\{ 3,  4,  6, 12\},\{ 4,  5,  7, 13\},\{ 1,  5,  6,  8\},\{ 2,  6,  7,  9\},\{ 3,  7,  8, 10\},\\
{ } & \{ 4,  8,  9, 11\},  \{ 5,  9, 10, 12\},\{ 6, 10, 11, 13\},\{ 1,  7, 11, 12\},\{ 2,  8, 12, 13\},\\
{ } &\{ 1,  3,  9, 13\}, \{ 1,  2,  4, 10\}\} .
\end{eqnarray*}

	Setting $C:=C_{(4,13)}\setminus V$,  we construct a C-Steiner configuration $X_C$   in $\mathbb P^4$. We have
	$$h_{X_C}= (1, 4, 10,  20,  35,  56,  84,   120,  165,   207).$$




We have  $\beta_{0,9}= 13,$  $\beta_{0,10}= 234,$ $\beta_{1,11}=702,$  $\beta_{2,12}= 663,$  $\beta_{3,13}=207.$ In particular, we have $\alpha(I_{X_C})=9$ and $\omega(I_{X_C})= \reg(I_{X_C})=10$. 

\end{example}

In general,  the following question is open.
\begin{question}Let $X_C\subseteq \mathbb P^n$ be a C-Steiner configuration of points of type $S(t,n,v)$.
 Do the graded Betti numbers and the Hilbert Function of $I_{X_C}^{(m)}$ only depend on the parameters $(t,n,v)$?
\end{question}

 We don't have any formula to describe the Hilbert function of $I_{X_C}^{(m)}$. 
For $m=1$ Proposition \ref{p.h-vector} implies that the $h$-vector of $X_C$ is a pure $O$-sequence (see \cite{MNZ} for all the terminology and background on this topic). This is not always true if $m>1$.  Using \cite{Cocoa} and \cite{Macaulay2}, we can show examples where the $h$-vectors could be not unimodal or not differentiable.

\begin{example}	Consider a Steiner system $(V_1,B_1)$  of type $S(2,3,7)$. Set $X_{C_1}:=X_{C_{(3,7)}\setminus B_1}$.  
	Using \cite{Macaulay2}, we found that  $I_{X_{C_1}}^{(7)}$ has the following $h$-vector
	$$h_{I_{X_{C_1}}^{(7)}}=(1, 3, 6, \ldots, 153_{16}, 171, \textbf{183, 182, 189}, 175, 140, 119, 84, 63, 42, 21, 14_{28} )$$
	that is a non-unimodal sequence (the index denotes the degree in which the dimension occurs, for instance, above the entry 153 is in degree 16). 
\end{example}

\begin{example}
	Consider a Steiner configuration $(V_2,B_2)$ of type $S(2,3,9)$ and set $X_{C_2}:=X_{C_{(3,7)}\setminus B_2}.$
Using \cite{Macaulay2},  we found that the $h$-vector of ${I_{X_{C_2}}^{(10)}}$ is
	$$h_{I_{X_{C_2}}^{(10)}}=(1,3, 6, \ldots, 528_{31}, 561, \textbf{583, 585, 603},  621, 639, 621, 567, 540, 540, 528, 468, 396,$$
	$$ 360, 360, 324, 252,216,216, 204, 144, 108, 108, 108, 72, 36, 36, 36, 24_{60})$$ 
	that is unimodal but not differentiable (the positive part of the first difference is not an O-sequence).

	
\end{example}


 We end this section with the following questions on $X_{\mathcal H,B}$ suggested by  experimental evidences using  \cite{Cocoa} and \cite{Macaulay2}. From a combinatorial point of view, two Steiner systems having the same parameters could have very different properties. We have examples where such differences effect the homological invariants.
	
	\begin{question}\label{Q1}
		Let $(V,B)$ be a Steiner system of type $S(t,n,v)$, and  $X_{\mathcal H,B}$ the associated Steiner configuration of points.  Assume that the hyperplanes in $\mathcal H$ are chosen generically.  Do the Hilbert function and the graded Betti numbers of $X_{\mathcal H,B}$ only depend on $t,n,v$? 
	\end{question}
	
	\begin{question}\label{Q2}
		Let $(V,B)$ be a Steiner system of type $S(t,n,v)$, and  $X_{\mathcal H,B}$ the associated Steiner configuration of points. Assume that the hyperplanes in $\mathcal H$ are chosen generically.  Are the Hilbert function and the graded Betti numbers of $X_{\mathcal H,B}$ \textit{generic with respect to the Hilbert function?} (i.e. the same as a set of $|B|$ generic points in $\mathbb P^n$?) 
	\end{question}

\section{Application to coding theory}\label{s.code theory}
In this section we show an application of the previous results on Steiner and C-Steiner configurations of points to coding theory. We recall the basic notion on linear coding and we compute the parameters of a linear code associated to a Steiner configuration and a C-Steiner configuration of points in $\mathbb P^n$. We refer to \cite{Coop-Gu,Toh3,Toh1,Toh2,TohVT,TohX} for a detailed initial motivation to study the connections between the minimum distance and some invariants coming from commutative/homological algebra. There are several ways to compute the minimum distance. One of them comes from linear algebra.

Let $k$ be any field and $X=\{P_1, \ldots, P_r \}\subseteq \mathbb P^n$ a not degenerate finite set of reduced points. 
 The linear code associated to $\Gamma$  denoted by $\mathcal C(X)$ is the image of the injective linear map $\varphi : k^{n+1}\to k^{r}.$ 

We are interested in three  parameters $[|X|,k_{X}, d_{X} ]$ that we use to evaluate the goodness of a linear code.
The first  number $|X|$ is the cardinality of $X$. The number $k_{X}$ is the dimension of the code as $k$-linear vector space, that is the rank of the matrix associated to $\varphi$. 
The number $d_{X}$ denotes the minimal distance of $\mathcal C(X)$, that is the minimum of the Hamming distance of two elements in $\mathcal C(X)$.
The \textit{Singleton bound} gives always an upper bound for this number: $d_{X}\le |X|-n$. When $d_X=|X|-n$ the code is  called a maximal distance separable code (MDS code).

The linear code associated to $X$ has generating matrix of type $(n+1) \times r$
\[A(X)= [ c_1  \ldots c_r ] \]
where $c_i$ are the coordinates of $P_i.$ Then the linear code $\mathcal C(X)$ has parameters $[|X|, Rank(A(X)), d_{X}]$. 
Assuming that $A(X)$ has no proportional columns is equivalent to say that the points $P_i$ are distinct points in $\mathbb P^n.$  Then $|X|=r,$  $Rank(A(X))=n+1$ and $r - d_X$ is the maximum number of these points that fit in a hyperplane of  $\mathbb P^n.$  
Remark 2.2 in \cite{Toh1,Toh2} says that the minimum distance $d_X$  is also the minimum number such that $r-d_X$ columns in $A(X)$ span an $n$-dimensional space. 
The generating matrix $A(X)$ of an $[|X|, n+1, d]$-linear code $\mathcal C$ naturally determines a matroid $M(\mathcal C).$


Denoted by $hyp(X)$ the maximum number of points contained in some hyperplane, $d_{X}$ has also geometrical interpretation, that is $d_{X}=|X|-hyp(X)$, see Section 2 in \cite{Toh1} and Remark 2.7 in \cite{TohVT}. In particular, the authors borrowed this terminology from coding theory since $d_X$ is exactly the minimum distance of the (equivalence class of) linear codes with generating matrix having as columns the coordinates of the points of $X$.

Set $X_B:=X_{\mathcal H,B}.$  We apply the results of the previous sections to compute the parameters of  linear codes associated to  both a Steiner configuration $X_B$ of points and its Complement $X_C$.

\begin{proposition}\label{hyp}  Let  $(V,B)$ be a  Steiner system of type  $S(t,n,v)$  with $|V|=v$.  If  $X_B$ is the Steiner configuration of points and $X_C$ its Complement, we have 
$$hyp(X_B)=\dfrac{{v-1\choose t-1}}{{n-1\choose t-1}}$$ and 
$$hyp(X_C)={v-1\choose n-1}-\dfrac{{v-1\choose t-1}}{{n-1\choose t-1}}.$$ 
\end{proposition}
\begin{proof}
Let $H_1\in \mathcal H$ be one of the hyperplanes involved in the construction of a Steiner configuration. Then it is clear by definition that the number of blocks in $B$  containing 1, i.e., the number of points in $H_1$ is  $$hyp(X_B)=\dfrac{{v-1\choose t-1}}{{n-1\choose t-1}}.$$
	Since the star configuration $X_{\mathcal H, C_{(n,v)}}$ has ${v-1\choose n-1}$ points lying on  $H_1$, we get 	$$hyp(X_C)={v-1\choose n-1}-\dfrac{{v-1\choose t-1}}{{n-1\choose t-1}}.$$ 
\end{proof}

\begin{proposition}\label{p.hyp e d} Let  $(V,B)$ be a  Steiner system of type  $S(t,n,v)$  with $|V|=v$. Then a Steiner configuration of points $X_B\subseteq \mathbb P^n$ defines a linear code with 	$$d_{X_B}= \dfrac{{v\choose t}}{{n\choose t}}- \dfrac{{v-1\choose t-1}}{{n-1\choose t-1}}.$$

A C-Steiner configuration of points $X_C\subseteq \mathbb P^n$  defines a linear code with
$$d_{X_C}= {v\choose n}-\dfrac{{v\choose t}}{{n\choose t}}-{v-1\choose n-1}+ \dfrac{{v-1\choose t-1}}{{n-1\choose t-1}}.$$

\end{proposition}
\begin{proof}
	The statement follows from Proposition \ref{hyp} since, from Remark 2.7 in \cite{TohVT}, for a set of points $X$ it is $d_{X}=|X|-hyp(X).$
\end{proof}

\begin{remark}
For a Steiner triple system $S(2, 3, v)$ with block set $B$, define $C:=C_{(3,v)}\setminus B$. Then $X_B$ and  $X_C$  have the following minimal distance
$$d_{X_B}= \dfrac{{v\choose 2}}{3}- \dfrac{v-1}{2}=\dfrac{(v-1)(v-3) }{6}$$
and
		$$d_{X_C}={v\choose 3}-\dfrac{{v\choose 2}}{3}-{v-1\choose 2}+ \dfrac{v-1}{2}= \dfrac{(v-1)(v-3)^2 }{6}.$$
\end{remark} 

With the above results, we have
\begin{theorem} \label{codici} Let  $(V,B)$  be a Steiner system $S(t,n,v)$ with $|V|=v$.   Then
\begin{enumerate}
\item the parameters of the linear code defined by a Steiner configuration of points $X_B$ are $\left [|B|,n+1, d_{X_B}\right];$
\item  the parameters of the linear code defined by a Complement of a Steiner configuration of points $X_{C}$ are $\left[ \binom{v}{n}-|B|,n+1, d_{X_{C}}\right ].$
\end{enumerate}
\end{theorem}

\begin{theorem} \label{codicMDS} 
 Let  $(V,B)$  be a Steiner system $S(t,n,v)$ with $|V|=v$.  
\begin{enumerate}
\item If  $n= \dfrac{{v-1\choose t-1}}{{n-1\choose t-1}}$ then $\mathcal C(X_{B})$ is a MDS code;
\item  if $n= {{v-1\choose t-1}}-\dfrac{{v-1\choose t-1}}{{n-1\choose t-1}}$ then $\mathcal C(X_{C})$ is a MDS code;
\end{enumerate}
\end{theorem}

\begin{example}

	Consider the Steiner system $S(2,3,7).$  The blocks are, up to isomorphism, $B:=\{\{1,2,3\},\{1,4,5\},\{1,6,7\}, \{2,4,6\}, \{2,5,7\}, \{3,4,7\}, \{3,5,6\}\}.$ 


Let $\ell_i:= x+2^iy+3^iz+5^iw\in \mathbb C[x,y,z,w]= \mathbb C[\mathbb P^3]$ be linear forms and let $H_i\subseteq \mathbb P^3$ be the hyperplane defined by $\ell_i$ for $i=1,\dots,7$. Set $\mathcal H :=\{H_1, \ldots, H_7 \}$. 
	One can check that any three hyperplanes in $\mathcal H$ meet in one point. 
	Computing with Cocoa \cite{Cocoa} the generating matrix of the linear code $\mathcal C(X_{\mathcal H, B})$ defined by the Steiner configuration on $B$ we get

{\tiny	\[
A(X_{\mathcal H, B}):=	\left( \begin{array}{ccccccc} 
	-15 & -1983 & -438045 & -350 & -639000 & 9315 & 104625 \\
	20 & 1576 & 269060 & 160 & 240075 & -2610 & -25875 \\
	-10 & -418 & -34230 & -35 & -37550 & 470 & 4250 \\
	1 & 17 & 523 & 1 & 666 & -9 & -99 \end{array}\right) 
	\]}

\noindent where the  columns are the seven coordinates  (among the thirty-five) of the intersection  points of any three hyperplanes $H_1,\dots,H_7$  corresponding to the seven  blocks $B.$
The parameters of the code $\mathcal C(X_{\mathcal H, B})$ are $[7,4,4].$

We note that only in this particular case, the linear code  $\mathcal C(X_{\mathcal H, B})$  associated to a Steiner system of type $S(2,3,7)$ has  $d_{X_B}=4=7-3$ and $\mathcal C(X_{\mathcal H, B})$ is a  maximal distance  separable code (MDS).

We now compute linear code $\mathcal C(X_{\mathcal H, C_{(3,7)}\setminus B})$ associated to the  Complement of the Steiner configuration. We get
{\tiny
\begin{equation*}\begin{split}
A(X_{\mathcal H, C_{(3,7)}\setminus B}):=&\left( \begin{array}{cccccccccccc} 
75 & -207 & 615 & 4845 & -465 & -70 & 27585 & -190 & 21300 & -5985 & 36273 & -3610 \\
-90 & 232 & -660 & -5060 & 460 & 64 & -23980 & 160 & -16005 & 4340 & -24728 & 2368 \\
40 & -94 & 250 & 1830 & -170 & -21 & 7190 & -45 & 3755 & -930 & 4586 & -387 \\
-3 & 5 & -9 & -43 & 11 & 1 & -239 & 1 & -111 & 19 & -115 & 7 \end{array}\right. \\
&\left. \begin{array}{ccccccccccc} 
225 & 1125 & -621 & 9225 & 6975 & 413775 & -29745 & -108819 & -3375 & -33750 & 5250 \\
-150 & -675 & 348 & -4950 & -3450 & -179850 & 11820 & 37092 & 1125 & 10125 & -1200 \\
50 & 200 & -94 & 1250 & 850 & 35950 & -2090 & -4586 & -250 & -2000 & 175 \\
-3 & -9 & 3 & -27 & -33 & -717 & 51 & 69 & 9 & 54 & -3 \end{array}\right. \\
&\left.\begin{array}{ccccc} 
-446175 & -101250 & 1012500 & -3138750 & 3037500 \\
88650 & 16875 & -151875 & 388125 & -253125 \\
-10450 & -2500 & 20000 & -42500 & 25000 \\
153 & 54 & -324 & 594 & -324 \end{array}\right)  \end{split}\end{equation*} 
	}

\noindent  where the columns are the twenty-eight coordinates of the  intersection  points (among the thirty-five) of any three hyperplanes $H_1,\dots,H_7$  corresponding to the twenty-eight  blocks{\tiny \begin{equation*} \begin{split}
C_{(3,7)}\setminus B=&\{\{1, 2, 4\},\{1, 2, 5\},\{1, 2, 6\},\{1, 2, 7\},\{1, 3, 4\},\{1, 3, 5\},\{1, 3, 6\},\{1, 3, 7\},\{1, 4, 6\},\\ &\{1, 4, 7\}, \{1, 5, 6\},\{1, 5, 7\},\{2, 3, 4\},\{2, 3, 5\},\{2, 3, 6\},\{2, 3, 7\},\{2, 4, 5\},\{2, 4, 7\},\\ &\{2, 5, 6\},\{2, 6, 7\}, \{3, 4, 5\},\{3, 4, 6\},\{3, 5, 7\},\{3, 6, 7\},\{4, 5, 6\},\{4, 5, 7\},\{4, 6, 7\},\{5, 6, 7\}\}.\end{split}\end{equation*}}
The parameters of the code $\mathcal C(X_{\mathcal H, C_{(3,7)}\setminus B})$ are $[28,4,16].$
 \end{example}

\end{document}